%% file: main.tex
\documentclass[letterpaper, 10 pt, conference]{ieeeconf}
\usepackage{common/sty/cdc}

\IEEEoverridecommandlockouts
\title{\LARGE \bf
Bilinear Koopman-Based Robust Model Predictive Control\\ for Unknown Nonlinear Systems via Contraction Metrics
}

\author{Yuki Higuchi and Kazuhiro Sato%
\thanks{This work was supported by JSPS KAKENHI Grant Number 23K28369 and 26K03232.}%
\thanks{Yuki Higuchi and Kazuhiro Sato are with Graduate School of Information Science and Technology,
        The University of Tokyo, Tokyo, 113-8656, Japan
        {\tt\small higuchi-yuki@g.ecc.u-tokyo.ac.jp, kazuhiro@mist.i.u-tokyo.ac.jp}}%
}

\begin{document}

\maketitle
\thispagestyle{empty}
\pagestyle{empty}

\begin{abstract}
  % Data-driven model predictive control (MPC) using Koopman operator theory has recently attracted attention for controlling unknown nonlinear systems. While linear Koopman realizations are commonly used due to their simplicity, bilinear Koopman realizations can provide significantly higher approximation accuracy for nonlinear control systems. However, robust MPC formulations that explicitly account for modeling errors in bilinear Koopman realizations remain limited.
  % In this work, we propose a robust MPC framework against modeling error based on data-driven bilinear Koopman realizations. A bilinear realization is first identified from data, and the approximation error is represented as a data-driven error set. To handle this error, we employ discrete-time robust control contraction metrics to construct a contraction-based tube around a nominal trajectory. The resulting formulation yields a time-varying tube whose radius evolves according to the contraction property, enabling systematic constraint tightening and robust constraint satisfaction.
  % Under the proposed formulation, we guarantee recursive feasibility and convergence to a neighborhood of the target state. Numerical experiments demonstrate robust stabilization of nonlinear systems and the advantages of the proposed method over existing Koopman MPC approaches in terms of performance.
  Data-driven model predictive control (MPC) using Koopman operator theory is a promising approach for constrained control of unknown nonlinear systems.
  While linear Koopman realizations are commonly used due to their simplicity, bilinear Koopman realizations can provide significantly higher approximation accuracy for nonlinear control systems. However, robust MPC (RMPC) formulations that account for modeling errors in bilinear Koopman realizations remain limited.
  This paper proposes a RMPC framework for unknown nonlinear systems with general nonlinear constraints based on data-driven bilinear Koopman realizations.
  A central difficulty is that finite-dimensional Koopman predictors need not preserve the manifold of valid lifted states, so multi-step prediction in lifted coordinates may leave the region where one-step error certificates apply.
  We address this issue by reprojecting each predicted lifted state back onto the manifold, thereby obtaining an error-aware discrete-time control-affine predictor in the original state space without impractical assumptions.
  For this predictor, we develop a discrete-time robust control contraction metric based homothetic tube construction, and then formulate a tube-based RMPC problem with terminal ingredients.
  Under the proposed formulation, we prove robust satisfaction of the original nonlinear constraints by the true closed-loop trajectory, recursive feasibility, and convergence to a neighborhood of the target state.
  Numerical experiments demonstrate robust stabilization of nonlinear systems and the advantages of the proposed method over existing Koopman-based RMPC approaches in terms of performance.
\end{abstract}

\section{Introduction}\label{section:introduction}
\input{./10_introduction/introduction.tex}

\textbf{Notation:}
For integers $a$ and $b$, we denote
$\integerset{a,b}=\{a,a+1,\ldots,b\}$.
Let $\mathbb S_+^n$ denote the set of $n\times n$ symmetric positive definite
matrices. We write $A\preceq B$ when $B-A$ is positive semi-definite.
For a vector $x$ and a symmetric positive definite matrix $M$, define
$\|x\|_M=\sqrt{x^\top Mx}$. The Cholesky factorization of $M$ is denoted by
$M=(M^{1/2})^\top M^{1/2}$.

\section{Preliminaries}\label{section:preliminary}

\input{./20_preliminary/preliminary.tex}

\section{Proposed Method}\label{section:proposed_method}
\input{./30_proposed_method/proposed_method.tex}

\section{Numerical Experiment}\label{section:numerical_experiment}
\input{./40_numerical_experiment/numerical_experiment.tex}

\section{Conclusion}\label{section:conclusion}
\input{./50_conclusion/conclusion.tex}
\balance
\bibliographystyle{ieeetr}
\bibliography{reference}

\end{document}

%% file: 10_introduction/introduction.tex
Model predictive control (MPC)~\cite{MRRS00} has become one of the standard methodologies for controlling constrained dynamical systems.
The implementation of MPC, however, requires a sufficiently accurate prediction model. In many applications, the governing dynamics are unknown, or identifying an accurate first-principles model is prohibitively expensive.
In robotics, soft robotic systems provide a representative example~\cite{ABR15}.
In such cases, a common approach is to first construct an approximate model from measured data and then use the learned model as the predictor in MPC\@.
System identification based on Koopman operator theory~\cite{Koopman31} has received significant attention as a means of constructing such data-driven predictors.
In Koopman theory, the system state is mapped to a high-dimensional space by a lifting map, which is typically a nonlinear function that maps the original state to real vectors in a higher-dimensional space, and the evolution of the lifting map is learned from data.
The most common approach is a linear realization of the lifted dynamics, which enables handling unknown nonlinear systems as linear systems.
However, this approach can suffer from insufficient approximation accuracy~\cite{KY22a}. To address this issue, bilinear Koopman realizations have been studied~\cite{Williams+16, Surana16}. It has been shown that, with a sufficiently high-dimensional lifting map and enough data, bilinear Koopman realizations can approximate control-affine systems with arbitrary accuracy, and experiments also indicate better approximation performance than linear realizations even for general nonlinear control systems in the finite-data and finite-dimensional lifting-map setting~\cite{BFV20}.
In practice, however, modeling errors of the learned Koopman realization are inevitable due to finite data and a finite-dimensional lifting map, and these errors can be viewed as uncertainties in the prediction model. MPC methods that explicitly account for such uncertainties are commonly referred to as robust MPC (RMPC).

While Koopman-based RMPC methods have been developed primarily for linear realizations~\cite{Zhang+22,MCV22,WLC23,ZLZH24,KTZS25,JBSL24, Chen+25}, a variety of bilinear Koopman-based MPC methods have been proposed~\cite{KY22a, KY22b, KY23, POR20, FB21, NHK23, Worthmann+24, Schimperna+25, Xiong+25, BGSW25}, and several of them also address approximation errors~\cite{KY22a, KY22b, KY23, Worthmann+24, Schimperna+25,Xiong+25,BGSW25}.
Nevertheless, to the best of our knowledge, no existing bilinear Koopman-based MPC simultaneously possesses all of the following properties:
robust satisfaction of the original state and input
constraints by the true system trajectory;
recursive feasibility of the MPC optimization problem;
convergence of the true state to a neighborhood of the target;
and freedom from the Koopman dictionary invariance assumption,
which will be explained in detail in the following paragraphs.
For example, constraint satisfaction for the true trajectory is not mathematically established in~\cite{KY22a, KY22b, KY23,Schimperna+25}, while recursive feasibility and convergence to a neighborhood of the target are not jointly established in~\cite{KY22a, KY22b, KY23, Xiong+25}. Moreover, some stability analyses rely
on the Koopman dictionary invariance assumption~\cite{Worthmann+24,BGSW25}, which is generally difficult to ensure when the underlying nonlinear dynamics are unknown.

One difficulty in achieving all of the above properties comes from multi-step prediction in the lifted space. The lifting map sends the original state space to a generally nonlinear manifold in the lifted coordinates.
Finite-dimensional data-driven realizations need not preserve it because of finite data and a finite-dimensional lifting map.
This loss of consistency is problematic for MPC\@. If a learned Koopman realization is used as a predictor in MPC and is propagated directly in lifted coordinates, the predicted lifted state may leave the manifold after one step. Later predictions then evaluate the learned realization at points that do not correspond to any original state. One solution is to assume the Koopman dictionary invariance, which ensures that the exact dynamics preserve the manifold, but selecting a lifting map that ensures this assumption is generally difficult for unknown nonlinear dynamics.
We can avoid this issue by projecting each predicted lifted state back onto the manifold after each step of prediction.
This yields a discrete-time control-affine model with uncertainties in the original state space, motivating the RMPC design for this class of systems.

In this context, control contraction metrics (CCM)~\cite{MS15} and robust CCM (RCCM)~\cite{Zhao+22} are powerful tools for analyzing stability and designing stabilizing controllers for control-affine systems. In particular, (R)CCM provide state-dependent metrics under which distances between closed-loop trajectories contract for (disturbed) control-affine systems.
Recent works also study the application of RCCM to MPC~\cite{SZK23, GSS24, ZS24}, leading to a tube-based MPC (TMPC) formulation that guarantees robust constraint satisfaction and convergence to a neighborhood of the target.
Unlike direct robust MPC formulations, which typically require min--max optimization, TMPC can be implemented by solving a single optimization problem at each time step.

Based on the above discussion, this paper proposes an RMPC framework that combines data-driven bilinear Koopman realizations with CCM-based TMPC\@.
First, we approximate the unknown nonlinear system using a bilinear Koopman realization~\cite{BFV20}, and estimate a modeling error set based on existing methods~\cite{Zhang+22,HKTA18}, thereby obtaining a discrete-time control-affine model with uncertainties.
The main contribution of this work is not tied to the particular modeling approach above, but lies in the control design and analysis for such models. Specifically, we develop a contraction-based TMPC framework for discrete-time control-affine systems with uncertainties. Building on existing RCCM-based TMPC~\cite{SZK23,GSS24,ZS24}, we formulate discrete-time RCCM tailored to such dynamics, and use it to construct a homothetic tube MPC scheme. Furthermore, we establish robust constraint satisfaction, recursive feasibility and convergence to a neighborhood of the target.
Numerical experiments demonstrate the effectiveness of the proposed method and show improved performance over existing approaches.

The remainder of this paper is organized as follows.
\Cref{section:preliminary} introduces problem setup, bilinear Koopman
realization, error-set estimation, and TMPC.
\Cref{section:proposed_method} derives discrete-time RCCM and presents the discrete-time RCCM-based TMPC
algorithm with theoretical analysis.
\Cref{section:numerical_experiment} provides numerical validation.

%% file: 20_preliminary/preliminary.tex
\subsection{Problem Setup}
We consider the unknown nonlinear control system
\begin{equation}
  \truesystem.
  \label{eq:true_system}
\end{equation}
Here, $\x\in\Xspace\subset\setR[\nx]$ is the state,
$\uinput\in\Uspace\subset\setR[\nuinput]$ is the input,
$\Xspace$ and $\Uspace$ are compact sets,
$\Freal:\setR[\nx]\times\setR[\nuinput]\to\setR[\nx]$ is an unknown
continuously differentiable nonlinear map describing true dynamics, and
$\xnext\in\setR[\nx]$ is the next state.
We aim to drive it to a target state $\xref$ with target input $\uref$ under
the compact constraint set $\Zsafe\subset \Xspace\times\Uspace$
given by
\[
  \Zsafe=\{(\x,\uinput)\in\setR[\nx]\times\setR[\nuinput]\mid
  h_j(\x,\uinput)\le 0,\ j\in\integerset{1,n_h}\},
\]
where each $h_j:\setR[\nx]\times\setR[\nuinput]\to\setR$, $j\in \integerset{1,n_h}$ is continuously
differentiable. We assume that the state $\x$ is measurable.
Based on measured data, we learn a model and design a
controller that satisfies $\Zsafe$ and steers the state $\x$ to $\xref$.

\subsection{Bilinear Koopman Realization}
\input{20_preliminary/koopman.tex}

\subsection{Modeling Error Set Estimation}
\input{20_preliminary/error_estimate.tex}

\subsection{Tube-Based MPC}
\input{20_preliminary/mpc.tex}

%% file: 20_preliminary/koopman.tex
We explain how to build a model from data using a bilinear Koopman
realization~\cite{BFV20}.
Define the lifting map
$\lift:\Xspace\times\Uspace\to\setR[\nz(\nuinput+1)+\nuinput], \lift=[\psi_1 \ldots \psi_{\nz(\nuinput+1)+\nuinput}]^\top$ by
\begin{equation}\label{eq:lift_function}
  \psi_i =
  \left\{
  \begin{array}{lll}
    \z_i                                                 & i =        & 1, \ldots, \nz                  \\
    \z_{i - \nz} \cdot \pi_{ \uinput_1 }                 & i =        & \nz + 1, \ldots, 2\nz           \\
    \;\;\;\vdots                                         & \;\;\vdots &                                 \\
    \z_{i - \nz \nuinput} \cdot \pi_{\uinput_{\nuinput}} & i =        & \nz \nuinput + 1, \ldots,       \\
                                                         &            & \nz (\nuinput + 1)              \\
    \pi_{\uinput_{i - \nz (\nuinput + 1)}}               & i =        & \nz (\nuinput + 1) + 1, \ldots, \\
                                                         &            & \nz (\nuinput + 1) + \nuinput
  \end{array}
  \right.,
\end{equation}
where $\z:\Xspace\to\setR[\nz]$, $\z=[\z_1\ldots\z_{\nz}]^\top$ is a vector of
continuously differentiable observables depending only on the state $\x$, and
$\pi_{\uinput_j}\ (j\in\integerset{1,n_u})$ returns the $j$-th component of the input $\uinput$.

Given a dataset of $\K$ samples
$\{(\x^{(k)},\uinput^{(k)},\xnext^{(k)})\}_{k=1}^{\K}$, where $\xnext^{(k)}=\Freal(\x^{(k)},\uinput^{(k)})$ for all $k \in \integerset{1,\K}$, solve
\begin{equation}\label{eq:koopman_least_squares}
  \underset{\Koopman}{\minimize}\sum_{k=1}^\K
  \left\|\lift(\xnext^{(k)},\uinput^{(k)})-
  \Koopman^\top\lift(\x^{(k)},\uinput^{(k)})\right\|_2^2
\end{equation}
to obtain the best matrix approximation $\bar{\Koopman}$.
Define
\begin{equation}
  \bar{\Koopman}^\top=
  \begin{bmatrix} A & H_1 & \cdots & H_{\nuinput} & B \\ \vdots & \vdots & \vdots & \vdots & \vdots \end{bmatrix},
  \label{eq:koopman_operator_matrix}
\end{equation}
with $A\in\setR[\nz\times\nz]$, $B\in\setR[\nz\times\nuinput]$, and $H_i\in\setR[\nz\times\nz]\ (i\in\integerset{1,\nuinput})$.
Then the bilinear Koopman realization is
\begin{equation}
  \begin{array}{ll}
    \xlift_{0}   & = \z(\x(0))                                                                \\
    \xlift_{t+1} & = A\xlift_t + \sum_{i=1}^{\nuinput}{\uinput(t)}_iH_i\xlift_t + B\uinput(t) \\
    \predictx_t  & = C\xlift_t
  \end{array}
\end{equation}
where $\xlift_t$ is the lifted state, $\predictx_t$ is the estimated state, $\x(0)$ is the initial state of the real system,
and $C\in\setR[\nx\times\nz]$ is obtained via
\begin{equation}\label{eq:koopman_output_least_squares}
  \underset{\tilde{C}}{\minimize}\sum_{k=1}^\K
  \left\|\x^{(k)}-\tilde{C}\z(\x^{(k)})\right\|_2^2.
\end{equation}
This learning procedure is summarized as follows.
\begin{algorithm}[H]
  \caption{Learning a Bilinear Koopman Realization}
  \label{algorithm:koopman_model_learning}
  \begin{algorithmic}[1]
    \REQUIRE Number of data points $\K$, lifting map $\lift$
    \ENSURE Coefficient matrices $A,\{H_i\}_{i=1}^{\nuinput},B,C$
    \STATE Collect dataset $\{(\x^{(k)},\uinput^{(k)},\xnext^{(k)})\}_{k=1}^\K$
    \STATE Solve \cref{eq:koopman_least_squares} to obtain $\bar{\Koopman}$
    \STATE Extract $A,\{H_i\}_{i=1}^{\nuinput},B$ from $\bar{\Koopman}$ by \cref{eq:koopman_operator_matrix}
    \STATE Solve \cref{eq:koopman_output_least_squares} to obtain $C$
  \end{algorithmic}
\end{algorithm}
Thus, we can denote the one-step prediction model as
\begin{equation}\label{eq:approx_system}
  \begin{aligned}
    \Fapprox(\x,\uinput)
     & =\fapprox(\x)+\gapprox(\x)\uinput                                          \\
     & =C\left(A\z(\x)+\sum_{i=1}^{\nuinput}{\uinput}_iH_i\z(\x)+B\uinput\right).
  \end{aligned}
\end{equation}
\begin{remark}
  In \cref{eq:koopman_operator_matrix}, only the first block row is used to define
  the bilinear dynamics of the lifted state $\z(\x)$. The lower block
  rows describe the evolution of the remaining lifted
  observables in $\lift(\x,\uinput)$, but they are not needed in the following discussion.
\end{remark}

%% file: 20_preliminary/error_estimate.tex
To design a robust controller against modeling errors,
we need to estimate possible modeling errors. We adopt an error set estimation method based on Hoeffding's inequality~\cite{Zhang+22,HKTA18}.

Given the learned model $\Fapprox$, first fix any candidate error set $\W$.
Then, collect a validation dataset
$\{(\x^{(k)},\uinput^{(k)},\xnext^{(k)})\}_{k=1}^{\K'}$ and define prediction error samples as
$\w^{(k)}=\xnext^{(k)}-\Fapprox(\x^{(k)},\uinput^{(k)})$. For the fixed set
$\W$, define $I(\w^{(k)})=\mathbf{1}(\w^{(k)} \in \W)$, where $\mathbf{1}$ is the indicator function.
Here, $\x^{(k)},\uinput^{(k)}$ are assumed to be sampled independently from the same
distribution, instead of being collected along a single trajectory, as
illustrated in \cref{figure:dataset_sampling}.

\begin{figure}[b]
  \centering
  \includegraphics[width=0.95\linewidth]{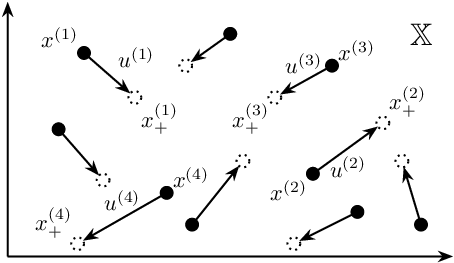}
  \caption{Example of independently sampled transition data for error-set estimation.}
  \label{figure:dataset_sampling}
\end{figure}

\begin{lemma}[Hoeffding's inequality~\cite{LS08}]
  If $I(\w^{(k)})\ (k\in\integerset{1,\K'})$ are i.i.d.\ random variables with
  $0\le I(\w^{(k)})\le 1$, then for any $\epsilon>0$, it holds that
  \[
    \Pr\left[\left|\frac{1}{\K'}\sum_{k=1}^{\K'}I(\w^{(k)})-\mathbb{E}[I(\w)]\right|\ge\epsilon\right]
    \le 2\exp(-2\K'\epsilon^2).
  \]
\end{lemma}

Hence, with confidence $1-\delta$, where $\delta=2\exp(-2\K'\epsilon^2)$,
\[
  \Pr[\w\in\W]=\Pr[I(\w)=1]\ge
  \frac{1}{\K'}\sum_{k=1}^{\K'}I(\w^{(k)})-\epsilon.
\]
holds.
This leads to the following algorithm to construct $\W$ satisfying target
confidence $1-\delta$ and coverage level $p$.

\begin{algorithm}[H]
  \caption{Modeling Error Set Estimation}
  \label{algorithm:error_set_estimation}
  \begin{algorithmic}[1]
    \REQUIRE Confidence $1-\delta$, coverage level $p$, approximate model $\Fapprox$, number of data points $\K'$
    \ENSURE Modeling error set $\W$
    \STATE Construct a candidate error set $\W$
    \STATE Collect dataset $\{(\x^{(k)},\uinput^{(k)},\xnext^{(k)})\}_{k=1}^{\K'}$
    \STATE Compute prediction errors $\w^{(k)}=\xnext^{(k)}-\Fapprox(\x^{(k)},\uinput^{(k)})$
    \STATE Accept $\W$ if $\dfrac{1}{\K'}\sum_{k=1}^{\K'}I(\w^{(k)})-\sqrt{-\dfrac{\ln(\delta/2)}{2\K'}}\ge p$
    \STATE Otherwise, reject $\W$ and go to Step 1
  \end{algorithmic}
\end{algorithm}

\begin{remark}\label{remark:error_set_construction}
  The algorithm above does not prescribe how to construct the candidate error set $\W$ in Step 1.
  In this work's numerical experiment, we construct a hyper-rectangle
  $\W=\{\w\in\setR[\nx]\mid |w_i|\le r_i, i\in\integerset{1,\nx}\}$
  by selecting component-wise error bounds from an independently generated
  construction dataset
  $\{(\x_{\mathrm c}^{(k)},\uinput_{\mathrm c}^{(k)},{\xnext}_{\mathrm c}^{(k)})\}_{k=1}^{K_{\mathrm c}}$ which is a dataset different from that used in the validation step of \cref{algorithm:error_set_estimation}.
  Let
  $\w_{\mathrm c}^{(k)}={\xnext}_{\mathrm c}^{(k)}-\Fapprox(\x_{\mathrm c}^{(k)},\uinput_{\mathrm c}^{(k)})$
  be the corresponding construction errors.
  For each state dimension $i\in\integerset{1,\nx}$, we enumerate the candidates
  $r_i\in\{0\}\cup\{|w_{\mathrm c,i}^{(k)}|:k\in\integerset{1,K_{\mathrm c}}\}$ and
  choose the smallest $r_i$ satisfying
  \[
    p+\nu \le \frac{1}{K_{\mathrm c}}\sum_{k=1}^{K_{\mathrm c}}
    \mathbf{1}(|w_{\mathrm c,i}^{(k)}|\le r_i)
    -\sqrt{\frac{\ln(2/\delta)}{2K_{\mathrm c}}}.
  \]
  Here $\nu>0$ denotes a small safety margin introduced in the component-wise
  construction stage. Thus, the bounds $r_i$ are selected separately for each
  state component when constructing $\W$. However, the resulting
  hyper-rectangle is evaluated in the validation step of
  \cref{algorithm:error_set_estimation} as a single set, so the final guarantee
  concerns the joint coverage probability $\Pr[\w\in\W]$ rather than the
  marginal coverage of each component interval.
\end{remark}

Using the approximate model and estimated error set, we define the disturbed
model
\begin{equation}
  \xnext=\fW(\x,\uinput,\d)\coloneq
  \Fapprox(\x,\uinput)+\Eapprox(\x,\uinput)\d,
  \label{eq:disturbed_system}
\end{equation}
where $(\x,\uinput)\in\Zsafe$, $\d\in\D$,
$\Eapprox:\setR[\nx]\times\setR[\nuinput]\to\setR[n_x\times n_d]$, and
$\D\subset\setR[n_d]$.
When the error set $\W$ is obtained by \cref{algorithm:error_set_estimation},
we can set $\Eapprox(\x,\uinput)=1$ and $\D=\W$.
However, we keep the generalized form~\eqref{eq:disturbed_system}
in what follows.

\begin{remark}
  In this work, we estimate the modeling error set in a probabilistic manner
  using Hoeffding's inequality. Under additional assumptions, however,
  recent works~\cite{Strasser+25,Strasser+26} provide state- and input-dependent
  proportional error bounds for bilinear Koopman-based surrogate models.
  Extending such bounds to the present setting is an interesting direction for future work.
\end{remark}

\begin{remark}\label{remark:probabilistic_error_bound}
  The error bound holds only probabilistically, meaning that the right-hand side
  of~\cref{eq:disturbed_system} does not strictly represent the true dynamics.
  Furthermore, this probabilistic guarantee is established on the validation
  dataset, rather than along the true trajectory of the closed-loop system.
  This limitation is common to many existing Koopman-based robust control
  methods~\cite{Worthmann+24,Schimperna+25,Xiong+25,BGSW25, Zhang+22, WLC23,KTZS25, JBSL24},
  which also treat uncertainty probabilistically, although only few of them
  provide deterministic robustness guarantees.
  Recent research has explored kernel-based methods for deriving deterministic
  error bounds~\cite{Strasser+25}, which represent an important direction for
  future work.
\end{remark}

%% file: 20_preliminary/mpc.tex
MPC that guarantees constraint satisfaction for models with
uncertainties as in \cref{eq:disturbed_system} is called robust MPC
(RMPC). RMPC at time $k$ proceeds as follows:
\rmpcenumerate{\fW}
Here $\x(k)$ is the observed state of the real system at time $k$.

At each control step, RMPC predicts future trajectories over the horizon $\N$ and
solves an optimization problem against worst-case uncertainties to find a robustly optimal control policy.
However, solving this min-max problem directly is often computationally expensive\@.

A practical alternative is tube-based MPC (TMPC).
TMPC first defines a nominal trajectory without uncertainties as
$\nox_{k+1}=\fW(\nox_k,\nou_k,0)$ for all $k\ge0$.
Here $\nox_k\in\Xspace$ and $\nou_k\in\Uspace$ are nominal state and input.
Then, a tube $\tube_{\nox_k}$ and a feedback law
$\kappa:\Xspace\times\Zsafe\to\Uspace$ are designed so that all possible
closed-loop real states remain inside the tube. They are designed to satisfy
\tubefeedbackcond{\fW}
Thus if the initial state belongs to $\tube_{\nox_0}$, all possible states controlled by $\kappa$ will
remain in $\tube_{\nox_k}$ for all $k\ge0$.

Next, nominal constraints ${\Zsafe}'_k$ are imposed so that
\nominalconst{useZ}
holds,
which guarantees
\tmpcprop{useZ}
In other words, if the initial state belongs to the initial tube and the
nominal trajectories satisfy tightened constraints, then all
real trajectories constrained by the feedback law will satisfy the original constraints.
Based on this, TMPC runs as
\tmpcenumerate{eq:disturbed_system}{\fW}{useZ}
The policy $\pi(\x(k'),k')=\kappa(\x(k'),\noxt{k'-k},\nout{k'-k})$
is feasible for the original RMPC problem while reducing computational burden
compared with direct min-max RMPC\@.

%% file: 30_proposed_method/proposed_method.tex
This section adapts discrete-time RCCM-based TMPC ideas to the disturbed
discrete-time systems in~\eqref{eq:disturbed_system}, building on CCM-based TMPC studies~\cite{SZK23,GSS24,ZS24}.

\subsection{Discrete-Time Robust Control Contraction Metric}
\input{30_proposed_method/drccm.tex}

\subsection{Tube and Feedback Design}
\input{30_proposed_method/tube_feedback.tex}

\subsection{Nominal Constraint Design}
\input{30_proposed_method/nominal_constraints.tex}

\subsection{Terminal Set Constraint}
\input{30_proposed_method/terminal_set.tex}

\subsection{TMPC Framework}
\input{30_proposed_method/mpc_framework.tex}

%% file: 30_proposed_method/drccm.tex
We first introduce the discrete-time robust control contraction metric used in
this work.
\begin{assumption}\label{assump:drccm}
  There exist a smooth matrix function $M:\setR[\nx]\to\mathbb S_+^{\nx}$,
  a continuous function $K:\setR[\nx]\to\setR[\nuinput \times \nx]$, and constants
  $0<\rho_c<1$, $\alpha_1>0$, $\alpha_2>0$ such that for all
  $(\x,\uinput)\in\Zsafe$, $\d\in\D$,
  \begin{subequations}\label{eq:drccm_conditions_full}
    \begin{align}
      \Acl(\x,\uinput,\d)^\top M(\xnext)\Acl(\x,\uinput,\d)
                              & \preceq (1-\rho_c)M(\x), \\
      \alpha_1 I\preceq M(\x) & \preceq \alpha_2 I,
      \label{eq:drccm_conditions}
    \end{align}
  \end{subequations}
  where $\Acl(\x,\uinput,\d)=\left.\frac{\partial\fW}{\partial\x}\right|_{(\x,\uinput,\d)}
    +\left.\frac{\partial\fW}{\partial\uinput}\right|_{(\x,\uinput,\d)}K(\x), \xnext=\fW(\x,\uinput,\d)$.
\end{assumption}

\begin{remark}\label{rem:sos}
  With the change of variables $W=M^{-1}$,
  \Cref{eq:drccm_conditions_full} can be transformed into linear matrix
  inequalities and solved via sum-of-squares formulations when $\Eapprox$ is control-affine~\cite{Zhao+22,SZK23,WMB21,MS15}.
  In particular, if the Koopman observables are chosen
  as polynomials in~\eqref{eq:lift_function}, then $\fW$ and $\Acl$ become polynomial functions of the state
  and input. By restricting $W$ and $K$ to polynomial functions, the above
  conditions can be converted exactly into a sum-of-squares formulation
  without additional approximation.
\end{remark}
% \begin{remark}
%   \cref{assump:drccm} requires the existence of a contraction-based incremental stabilizing feedback for the learned model with the estimated modeling error. This assumption is reasonable when the modeling error is sufficiently small, because the robust contraction condition then approaches the standard CCM condition as the discretization step goes to zero.
%   In particular, for the nominal case without disturbance, CCM existence is guaranteed for feedback linearizable systems~\cite{MS15}. Hence, when the learned model approximates the true dynamics sufficiently accurately, \cref{assump:drccm} is expected to hold at least for system classes that admit a nominal CCM, such as feedback linearizable systems.
% \end{remark}

%% file: 30_proposed_method/tube_feedback.tex
Next, we design the tube and the feedback law.
For any $\x,\nox\in\setR[\nx]$, let $\Gamma(\nox,\x)$ be the set of
component-wise smooth curves $\gamma:[0,1]\to\setR[\nx]$ satisfying
$\gamma(0)=\nox$, $\gamma(1)=\x$.
Define $\riemannenergyimpl$,
and denote a minimizer by the geodesic $\gamma^*$, which exists under
\cref{eq:drccm_conditions}~\cite{MS15}.
For any $\x,\nox\in\setR[\nx]$, set $\gamma^u(s)=\nou+\int_0^s K(\gamma^*(s'))\dot{\gamma}^*(s')\,ds', \kappa(\x,\nox,\nou) =\gamma^u(1)$.

\begin{proposition}\label{prop:tube_dynamics}
  Suppose \cref{assump:drccm} holds. Then, for any
  $\x,\nox\in\setR[\nx],\nou \in \setR[\nuinput]$
  satisfying
  $(\gamma^*(s),\gamma^u(s))\in\Zsafe$ for all $s\in[0,1]$, it holds that
  \[
    V(\xnext,\noxnext)\le\sqrt{1-\rho_c}\,V(\x,\nox)+L_{\Eapprox},
  \]
  where
  $
    L_{\Eapprox}=\max_{(\x,\uinput)\in\Zsafe,\d\in\D}
    \|\Eapprox(\x,\uinput)\d\|_{M(\fW(\x,\uinput,\d))},
  $
  $\xnext=\Freal(\x,\uinput)$,
  $\noxnext=\Fapprox(\nox,\nou)$, and
  $\uinput=\kappa(\x,\nox,\nou)$.
\end{proposition}

\begin{proof}
  By the definition of the error set $\D$, there exists $\d\in\D$ such that
  $\xnext=\fW(\x,\uinput,\d)$.
  Define
  $c^+(s)=\fW(\gamma^*(s),\gamma^u(s),s\d)$, $c^+(0)=\fW(\nox,\nou,0)=\Fapprox(\nox,\nou)=\noxnext$ and
  $c^+(1)=\fW(\gamma^*(1),\gamma^u(1),\d)=\xnext$. Hence,
  $c^+\in\Gamma(\noxnext,\xnext)$.
  Differentiating $c^+(s)$ with respect to $s$ gives
  $
    \dot c^+(s)=\Acl(\gamma^*(s),\gamma^u(s),s\d)\dot\gamma^*(s)
    +\Eapprox(\gamma^*(s),\gamma^u(s))\d
  $.
  Combining this with \cref{assump:drccm}, for all $s\in[0,1]$, we have
  $
    \|\dot c^+(s)\|_{M(c^+(s))}
    \leq \|\Acl(\gamma^*(s),\gamma^u(s),s\d)\dot\gamma^*(s)\|_{M(c^+(s))}+\|\Eapprox(\gamma^*(s),\gamma^u(s))\d\|_{M(c^+(s))}
    \leq \sqrt{1-\rho_c}\,\|\dot\gamma^*(s)\|_{M(\gamma^*(s))} +\|\Eapprox(\gamma^*(s),\gamma^u(s))\d\|_{M(c^+(s))}
  $.
  Integrating both sides over $s\in[0,1]$ yields $V(\xnext,\noxnext)\leq\sqrt{1-\rho_c}\,V(\x,\nox)+L_{\Eapprox}$.
\end{proof}

\begin{remark}
  The constant $L_{\Eapprox}$ can be computed together with the discrete-time RCCM in the
  synthesis step described in \cref{rem:sos}, following the same
  procedure used for CCM- and RCCM-based designs~\cite{SZK23,Zhao+22}.
\end{remark}

The geodesic can be computed numerically via a Chebyshev pseudospectral method with proper discretization~\cite{LM17}.

%% file: 30_proposed_method/nominal_constraints.tex
We next tighten the constraints for the nominal system.

\begin{proposition}[\cite{SZK23}]\label{prop:nominal_constraint}
  Under \cref{assump:drccm}, for any
  $\x,\nox\in\setR[\nx] ,\nou \in \setR[\nuinput]$
  satisfying
  $h_j(\nox,\nou)+c_jV(\x,\nox)\le0,\ \forall j\in\integerset{1,n_h}$,
  $(\gamma^*(s),\gamma^u(s))\in\Zsafe$ holds for all $s\in[0,1]$.
  Here
  $c_j=\max_{\nox,\nou}\left\|\left(\frac{\partial h_j}{\partial \x}+
    \frac{\partial h_j}{\partial u}K(\nox)\right)M(\nox)^{-1/2}\right\|$.
\end{proposition}

Using this proposition, we obtain the following theorem.

\begin{theorem}\label{theorem:tube_feedback}
  Under \cref{assump:drccm}, consider initial state $\x(0)$ and sequences
  $\nox_k,\nou_k,\delta_k$ satisfying
  \begin{subequations}
    \begin{align}
       & h_j(\nox_k,\nou_k)+c_j\delta_k\le0,
      \quad \forall j\in\integerset{1,n_h},
      \label{eq:nominal_constraint}                           \\
       & \nox_{k+1}=\Fapprox(\nox_k,\nou_k),
      \label{eq:nominal_dynamics}                             \\
       & V(\x(0),\nox_0)\le\delta_0,
      \label{eq:initial_tube}                                 \\
       & \delta_{k+1}=\sqrt{1-\rho_c}\,\delta_k+L_{\Eapprox}.
      \label{eq:tube_dynamics}
    \end{align}
  \end{subequations}
  % Then for any $k\in\{0,1,\ldots\}$ and $d_k\in\D$, the trajectory
  % $\x(k+1)=\fW(\x(k),\kappa(\x(k),\nox_k,\nou_k),d_k)$ satisfies
  Then for any $k\in\{0,1,\ldots\}$, the real trajectory $\x(k+1)=\Freal(\x(k),\kappa(\x(k),\nox_k,\nou_k))$ satisfies
  \begin{subequations}
    \begin{align}
      V(\x(k),\nox_k)                     & \le\delta_k,
      \label{eq:theorem_tube}                            \\
      (\x(k),\kappa(\x(k),\nox_k,\nou_k)) & \in\Zsafe.
      \label{eq:theorem_constraint}
    \end{align}
  \end{subequations}
\end{theorem}

\begin{proof}
  The proof proceeds by induction on $k$.
  For $k=0$, \cref{eq:theorem_tube} follows from \cref{eq:initial_tube}.
  Also, from \cref{eq:nominal_constraint}, \cref{eq:initial_tube}, and
  \cref{prop:nominal_constraint},
  $(\gamma^*(s),\gamma^u(s))\in\Zsafe$ for all $s\in[0,1]$;
  substituting $s=1$ gives \cref{eq:theorem_constraint}.
  Assume the claim holds at $k=\ell$.
  By \cref{eq:theorem_tube}, \cref{eq:tube_dynamics} at $k=\ell$, and
  \cref{prop:tube_dynamics},
  $V(\x(\ell+1),\nox_{\ell+1})
    \le\sqrt{1-\rho_c}\,V(\x(\ell),\nox_\ell)+L_{\Eapprox}
    \le\sqrt{1-\rho_c}\,\delta_\ell+L_{\Eapprox} =\delta_{\ell+1}$.
  So \cref{eq:theorem_tube} holds at $k=\ell+1$.
  Combining this with \cref{eq:nominal_constraint} at $k=\ell + 1$ and
  \cref{prop:nominal_constraint},
  $(\gamma^*(s),\gamma^u(s))\in\Zsafe$ for all $s\in[0,1]$ at $k=\ell+1$;
  substitute $s=1$ to obtain \cref{eq:theorem_constraint}.
  Hence both statements hold for all $k\ge0$.
\end{proof}

%% file: 30_proposed_method/terminal_set.tex
To guarantee recursive feasibility and convergence of nominal trajectories, we introduce a terminal set
constraint.

\begin{assumption}\label{ass:terminalset}
  There exist a terminal set
  $\liXf\subseteq\setR[\nx]\times\setR_{\ge0}$,
  a terminal controller $k_f:\setR[\nx]\to\setR[\nuinput]$, and a terminal cost
  $\ell_f:\setR[\nx]\to\setR_{\ge0}$ such that for any
  $(\nox,\delta)\in\liXf$ and for all $j \in \integerset{1,n_h}$,
  \begin{subequations}
    \begin{align}
       & (\noxnext,\deltanext)\in\liXf,
      \label{eq:terminal_set}                                 \\
       & h_j(\nox,k_f(\nox))+c_j\delta\le0,
      \ h_j(\noxnext,k_f(\noxnext))+c_j\deltanext\le0,
      \label{eq:terminal_constraint}                          \\
       & \ell(\nox,k_f(\nox))\le\ell_f(\nox)-\ell_f(\noxnext)
      \label{eq:terminal_cost}
    \end{align}
  \end{subequations}
  hold, where $\noxnext=\Fapprox(\nox,k_f(\nox))$ and
  $\deltanext=\sqrt{1-\rho_c}\delta+L_{\Eapprox}$.

\end{assumption}

The following proposition shows that if \cref{assump:drccm} is satisfied, \cref{ass:terminalset} is also likely to be satisfied.

\begin{proposition}\label{prop:terminalset}
  Under \cref{assump:drccm}, suppose there exists $\uref$ such that $\xrefeq$ and
  $h_j(\xref,\uref)+c_j\delta_f \le 0$ for all $j\in\integerset{1,n_h}$, where
  \begin{equation}\label{eq:delta_f}
    \delta_f = \frac{L_{\Eapprox}}{1-\sqrt{1-\rho_c}}.
  \end{equation}
  Then \cref{ass:terminalset} holds with
  \begin{subequations}
    \begin{align}
      \liXf=\{ & (\nox,\delta)\in\setR[\nx]\times\setR_{\ge0}\mid \notag \\
               & \nox=\xref, \label{eq:terminal_nox_prop}                \\
               & \delta\le\delta_f\},
      \label{eq:terminal_delta_prop}
    \end{align}
  \end{subequations}
  $k_f(\nox)\equiv\uref$, $\ell_f(\nox)\equiv 0$ and $\ell(\xref, \uref)=0$.
\end{proposition}

\begin{proof}
  Let $(\nox,\delta)\in\liXf$.
  From \cref{eq:terminal_nox_prop} and $k_f(\nox)\equiv\uref$,
  $\noxnext=\Fapprox(\nox,k_f(\nox))=\xref$, so
  \Cref{eq:terminal_nox_prop} is preserved at the next time step.
  Moreover, the definition of $\delta_f$ and \cref{eq:terminal_delta_prop} imply
  $\deltanext=\sqrt{1-\rho_c}\delta+L_{\Eapprox}
    \le\sqrt{1-\rho_c}\delta_f+L_{\Eapprox}=\delta_f$, so \cref{eq:terminal_delta_prop} is preserved at the next time step.
  Hence, $(\noxnext,\deltanext)\in\liXf$, and \cref{eq:terminal_set} holds.
  Furthermore, $h_j(\nox,k_f(\nox))+c_j\delta = h_j(\xref,\uref)+c_j\delta \le h_j(\xref,\uref)+c_j\delta_f \le 0$, and
  $h_j(\noxnext,k_f(\noxnext))+c_j\deltanext = h_j(\xref,\uref)+c_j\deltanext \le h_j(\xref,\uref)+c_j\delta_f \le 0$ for all $j \in \integerset{1,n_h}$, so \cref{eq:terminal_constraint} holds.
  Finally, since $\ell(\xref, \uref)=0$ and $\ell_f(\nox)\equiv 0$, \cref{eq:terminal_cost} holds.
\end{proof}

\begin{remark}
  The condition $\xrefeq$ means that $(\xref, \uref)$ is an equilibrium of the learned model~\eqref{eq:approx_system}.
  In general, if $(\xref,\uref)$ is chosen
  arbitrarily after constructing $\Fapprox$, this condition may not hold.
  A practical design is to choose $\xref$ first and then compute $\uref$ by
  \begin{equation}\label{eq:uref_design}
    \underset{\uinput}{\minimize}\left\|\xref-\Fapprox(\xref,\uinput)\right\|.
  \end{equation}
  If the minimum is sufficiently small, that pair can be adopted as the reference.
  Also, if observables are designed so that $\z(0)=0$ in
  \cref{eq:lift_function}, then $(\xref, \uref)=(0, 0)$ always satisfies $\xrefeq$.
\end{remark}

%% file: 30_proposed_method/mpc_framework.tex
Combining the above ingredients, we present the TMPC formulation and
its theoretical guarantees.
Given the state $\x(k)$ at time $k$, we solve the horizon-$\N$ problem:
\begin{subequations}\label{eq:mpc}
  \begin{align}
    V^*(\x(k))=
     & \min_{\nox_{\cdot|k},\nou_{\cdot|k},\delta_{\cdot|k}}
    \sum_{i=0}^{\N-1}\ell(\nox_{i|k},\nou_{i|k})+\ell_f(\nox_{\N|k}) \notag                         \\
    \subjectto\quad
     & \nox_{i+1|k}=\Fapprox(\nox_{i|k},\nou_{i|k}),\ \forall i\in\integerset{0,\N-1},
    \label{eq:nox}                                                                                  \\
     & \delta_{i+1|k}=\sqrt{1-\rho_c}\,\delta_{i|k}+L_{\Eapprox},\ \forall i\in\integerset{0,\N-1},
    \label{eq:delta}                                                                                \\
     & h_j(\nox_{i|k},\nou_{i|k})+c_j\delta_{i|k}\le0, \notag                                       \\
     & \hspace{3em}\forall j\in\integerset{1,n_h},\ \forall i\in\integerset{0,\N-1},
    \label{eq:constraint}                                                                           \\
     & V(\x(k), \nox_{0|k})\le\delta_{0|k},
    \label{eq:tubeinit}                                                                             \\
     & (\nox_{\N|k},\delta_{\N|k})\in\liXf
    \label{eq:terminal}
  \end{align}
\end{subequations}
with
\begin{equation}\label{eq:cost}
  \ell(\nox_{i|k},\nou_{i|k})=
  \|\nox_{i|k}-\xref\|_Q^2+\|\nou_{i|k}-\uref\|_R^2,
\end{equation}
where $Q\in \mathbb S_+^{\nx}$, $R\in \mathbb S_+^{\nuinput}$ are weighting
matrices.
The decision variables are the nominal state and input trajectories
$\nox_{i|k},\nou_{i|k}$ and the tube radius $\delta_{i|k}$. The nominal state $\nox_{i|k}$ satisfies the learned approximate dynamics \cref{eq:approx_system} (cf.~\cref{eq:nox}).
The tube radius $\delta_{i|k}$ evolves according to \cref{prop:tube_dynamics} (cf.~\cref{eq:delta}, \cref{eq:tubeinit}). The constraints on the nominal state and input \cref{eq:constraint} are designed based on \cref{theorem:tube_feedback} to ensure that the true state and input satisfy the original constraints. The terminal constraint \cref{eq:terminal} is also introduced based on \cref{ass:terminalset} to guarantee some properties of MPC\@.
If
$\{\nox_{i|k}^*,\nou_{i|k}^*,\delta_{i|k}^*\mid i\in\integerset{0,\N}\}$
is an optimal solution, the applied input is
\begin{equation}\label{eq:mpc_control_law}
  \uinput(k)=\kappa(\x(k),\nox_{0|k}^*,\nou_{0|k}^*).
\end{equation}
The offline and online procedures are as follows:
\begin{algorithm}[H]
  \caption{TMPC Offline Design}
  \label{algorithm:mpc_offline}
  \begin{algorithmic}[1]
    \REQUIRE Model with uncertainties $\fW$ \cref{eq:disturbed_system}
    \STATE Compute $M(\x),K(\x)$ and $\rho_c$ (\cref{assump:drccm})
    \STATE Compute $L_{\Eapprox},c_j$ (\cref{prop:tube_dynamics,prop:nominal_constraint})
    \STATE Design $\liXf$, $k_f$, and $\ell_f$ (\cref{ass:terminalset,prop:terminalset})
  \end{algorithmic}
\end{algorithm}

\begin{algorithm}[H]
  \caption{TMPC Online Control}
  \label{algorithm:mpc_online}
  \begin{algorithmic}[1]
    \REQUIRE Weighting matrices $Q,R$, horizon $\N$
    \FOR{each time step $k$}
    \STATE Solve \cref{eq:mpc}
    \STATE Apply \cref{eq:mpc_control_law} to the real system
    \ENDFOR
  \end{algorithmic}
\end{algorithm}
Finally, we prove recursive feasibility and convergence of nominal trajectories for this TMPC framework.

\begin{theorem}\label{thm:tmpc}
  If~\eqref{eq:mpc} is feasible at time $k=0$ for initial state $\x(0)$, then it
  is feasible for all $k\ge0$. Moreover,
  $\lim_{k\to\infty}\nox_{0|k}=\xref$ and
  $\lim_{k\to\infty}\nou_{0|k}=\uref$.
\end{theorem}

\begin{proof}
  Assume~\eqref{eq:mpc} is feasible at time $k$.
  For simplicity, we define $\nou^*_{\N|k}=k_f(\nox^*_{\N|k})$, and define
  $\nox^*_{\N+1|k},\delta^*_{\N+1|k}$ by propagating \cref{eq:nox,eq:delta} with $\nou^*_{\N|k}$.
  Construct a candidate at time $k+1$ by $\nou_{i|k+1}=\nou^*_{i+1|k},    \nox_{0|k+1}=\nox^*_{1|k},
    \delta_{0|k+1}=V(\x(k+1),\nox_{0|k+1})$
  for all
  $i\in\integerset{0,\N-1}$
  and propagate remaining terms by \cref{eq:nox,eq:delta}.
  Constraints \cref{eq:nox,eq:delta,eq:tubeinit} at time $k+1$ are immediate and $\nox_{i|k+1} = \nox^*_{i+1|k}$ for all $i\in\integerset{0,\N-1}$ holds by construction.
  From \cref{eq:delta,eq:tubeinit} at time $k$ and \cref{prop:tube_dynamics},
  $\delta_{0|k+1}=V(\x(k+1),\nox_{0|k+1})=V(\x(k+1),\nox^*_{1|k})\le\delta_{1|k}^*$, so
  $\delta_{i|k+1}\le\delta_{i+1|k}^*$ for all $i\in\integerset{0,\N-2}$.
  Hence,
  $h_j(\nox_{i|k+1},\nou_{i|k+1})+c_j\delta_{i|k+1} \le h_j(\nox^*_{i+1|k},\nou^*_{i+1|k})+c_j\delta^*_{i+1|k} \le 0$ hold for all $i \in \integerset{0,\N-2}$ and for all $j \in \integerset{1,n_h}$ by \cref{eq:constraint} at time $k$.
  For $i=\N-1$,~\eqref{eq:terminal} at time $k$ and \cref{ass:terminalset} imply
  $h_j(\nox_{\N-1|k+1},\nou_{\N-1|k+1})+c_j\delta_{\N-1|k+1}\le0$ for all $j \in \integerset{1,n_h}$, so~\eqref{eq:constraint} holds at time $k+1$
  and \cref{eq:terminal} at time $k+1$ also follows.
  Thus feasibility is recursive.
  Next, we prove convergence of nominal trajectories.
  It holds that
  \begin{align*}
    V^*(\x(k+1))
     & \le \sum_{i=0}^{\N-1}\ell(\nox_{i+1|k},\nou_{i+1|k})
    +\ell_f(\nox_{\N+1|k})                                  \\
     & =V^*(\x(k)) - \ell(\nox_{0|k},\nou_{0|k})
    +\ell_f(\nox_{\N+1|k})                                  \\
     & \phantom{=V^*(\x(k)) - }
    +\ell(\nox_{\N|k},\nou_{\N|k})-\ell_f(\nox_{\N|k})      \\
     & \le V^*(\x(k)) - \ell(\nox_{0|k},\nou_{0|k}),
  \end{align*}
  where the last inequality uses \cref{eq:terminal_cost}.
  Summing over $k$ shows $\sum_{k=0}^\infty\ell(\nox_{0|k},\nou_{0|k})<\infty$.
  By compactness of the constraints and continuity of the cost,
  $\underset{k\to\infty}{\lim}\ell(\nox_{0|k},\nou_{0|k})=0$.
  Because $\ell$ is defined as in~\eqref{eq:cost} and $Q\in \mathbb S_+^{\nx}$, $R\in \mathbb S_+^{\nuinput}$, it follows that $\nox_{0|k} \to \xref$, $\nou_{0|k} \to \uref$.
\end{proof}

The construction of \cref{eq:mpc} and the statement of \cref{thm:tmpc} lead to the following corollary.

\begin{corollary}\label{cor:state_convergence}
  Suppose \cref{assump:drccm} holds, and let the terminal ingredients
  $\liXf,k_f,\ell_f$ be chosen as in \cref{prop:terminalset}.
  If~\eqref{eq:mpc} is feasible at time $k=0$ for initial state $\x(0)$,
  then the resulting closed-loop trajectory $\x(k)$ satisfies
  \begin{equation}\label{eq:state_convergence}
    \limsup_{k\to\infty}\|\x(k)-\xref\|_2
    \le \frac{\delta_f}{\sqrt{\alpha_1}},
  \end{equation}
  where $\delta_f$ is given in \cref{prop:terminalset} and $\alpha_1$ is given in \cref{assump:drccm}.
\end{corollary}

\begin{proof}
  Let $\{\nox_{i|k}^*,\nou_{i|k}^*,\delta_{i|k}^*\mid i\in\integerset{0,\N}\}$ be an optimal solution to~\eqref{eq:mpc} at time $k$.
  For each $k$, the triangle inequality and \cref{eq:tubeinit} give
  $
    V(\x(k),\xref)
    \le V(\x(k),\nox_{0|k}^*) + V(\nox_{0|k}^*,\xref)
    \le \delta_{0|k}^* + V(\nox_{0|k}^*,\xref).
  $
  By recursively applying \cref{eq:delta} for $\N$ steps, we obtain
  $
    \delta_{\N|k}^*
    = \sqrt{1-\rho_c}^{\N}\delta_{0|k}^*
    + \frac{1-\sqrt{1-\rho_c}^{\N}}{1-\sqrt{1-\rho_c}}L_{\Eapprox}.
  $
  Applying \cref{eq:delta_f} in the same way, we have
  $
    \delta_f
    = \sqrt{1-\rho_c}^{\N}\delta_f
    + \frac{1-\sqrt{1-\rho_c}^{\N}}{1-\sqrt{1-\rho_c}}L_{\Eapprox}.
  $
  From \cref{eq:terminal,eq:terminal_delta_prop}, we obtain
  $\delta_{\N|k}^* \le \delta_f$, and therefore $\delta_{0|k}^* \le \delta_f$.
  % By \cref{eq:terminal} and \cref{prop:terminalset}, we have $\delta_{\N|k}^* \le \delta_f$ for all $k$, and~\eqref{eq:delta}
  % yields $\delta_{0|k}^* \le \delta_f$ for all $k$.
  Moreover, it holds that $\nox_{0|k}^* \to \xref$ as $k \to \infty$ by \cref{thm:tmpc}. Hence,
  $
    \underset{k\to\infty}{\limsup}\, V(\x(k),\xref) \le \delta_f.
  $
  Since $V(\x(k),\xref)\ge \sqrt{\alpha_1}\|\x(k)-\xref\|_2$ by \cref{eq:drccm_conditions},
  \cref{eq:state_convergence} follows.
\end{proof}

Therefore, the proposed TMPC framework drives the true state of the unknown system to a neighborhood of the target state while satisfying the constraints even under modeling errors.

%% file: 40_numerical_experiment/numerical_experiment.tex
We consider an inverted pendulum:
\begin{equation}
  \begin{bmatrix}
    \dot{x}_1 \\
    \dot{x}_2
  \end{bmatrix}
  =
  \begin{bmatrix}
    x_2 \\
    4g\sin x_1 - 3u\cos x_1
  \end{bmatrix}
\end{equation}
where $x=[x_1\ x_2]^\top\in\mathbb{R}^2$, with angle $x_1$~($\si{\radian}$) and angular
velocity $x_2$~($\si{\radian\per\second}$), $u\in\mathbb{R}$ is the control input in $\si{\radian\per\second\squared}$,
and the
gravitational acceleration is $g=\SI{9.81}{\meter\per\second\squared}$.
Constraints are
$-[1~\si{\radian}\ 2~\si{\radian\per\second}]^\top\le x\le[1~\si{\radian}\ 2~\si{\radian\per\second}]^\top$ and $-20~\si{\radian\per\second\squared} \le u \le 20~\si{\radian\per\second\squared}$.
The target is $\xref=[0\ 0]^\top$ from initial state
$\x(0)=[0.2\ 1]^\top$.

We discretize the continuous system with sampling time $T_s=\SI{0.0075}{\second}$.
For \cref{algorithm:koopman_model_learning}, we use observable set
\begin{align}
  \{\psi_i\}=\{(\x_1^{\rho_1}\cdots\x_{\nx}^{\rho_{\nx}})\hat{u}
   & \mid \rho_1+\cdots+\rho_{\nx}\le\rho, \notag \\
   & \hat{u}\in\{1,u_1,\ldots,u_{\nuinput}\}\}
\end{align}
with $\rho=3$ and $\K=10000$.
State-input samples are drawn uniformly from the constraint set.
For \cref{algorithm:error_set_estimation}, we set confidence $1-\delta=0.9$,
coverage level $p=0.9$, and $\K'=10000$.
For \cref{algorithm:mpc_online}, we use
$Q=\mathrm{diag}\{1,1\}$, $R=0.1$, and horizon $N=12$.
For \cref{prop:terminalset} and \cref{assump:drccm}, we obtain $\delta_f=0.111$ and $\alpha_1=0.0422$, which yield the ultimate bound
$\delta_f/\sqrt{\alpha_1}=0.540$ in \cref{eq:state_convergence}.

All computations were performed in MATLAB R2025b on macOS Tahoe 26.3
with Apple M4 Max MacBook Pro (16-core processor, \SI{128}{\giga\byte} memory).
The MPC problem was modeled in CasADi~\cite{Andersson+18} and solved by IPOPT~\cite{WB06}.

\Cref{fig:rmpc_trajectory} shows simulation results.
\begin{figure}[b]
  \centering
  \includegraphics[width=0.95\linewidth]{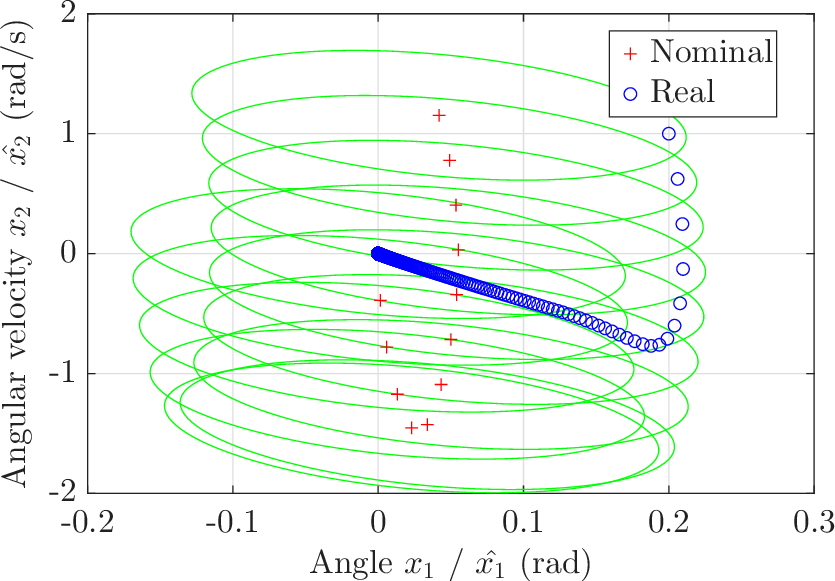}
  \caption{Closed-loop trajectory, nominal trajectory, and tubes under the proposed TMPC.}
  \label{fig:rmpc_trajectory}
\end{figure}
Circles denote real states, plus signs denote nominal states from the MPC optimization at the first step, and ellipses
indicate tubes around nominal states.
The state converges to the target and the area of the tubes satisfies the constraint, demonstrating the correctness of the proposed method.
\Cref{fig:inverted_pendulum_comparison} compares our method with the previous work on TMPC using linear Koopman realizations~\cite{Zhang+22}.
\begin{figure}[tb]
  \centering
  \vspace{0.5em}
  \includegraphics[width=0.95\linewidth]{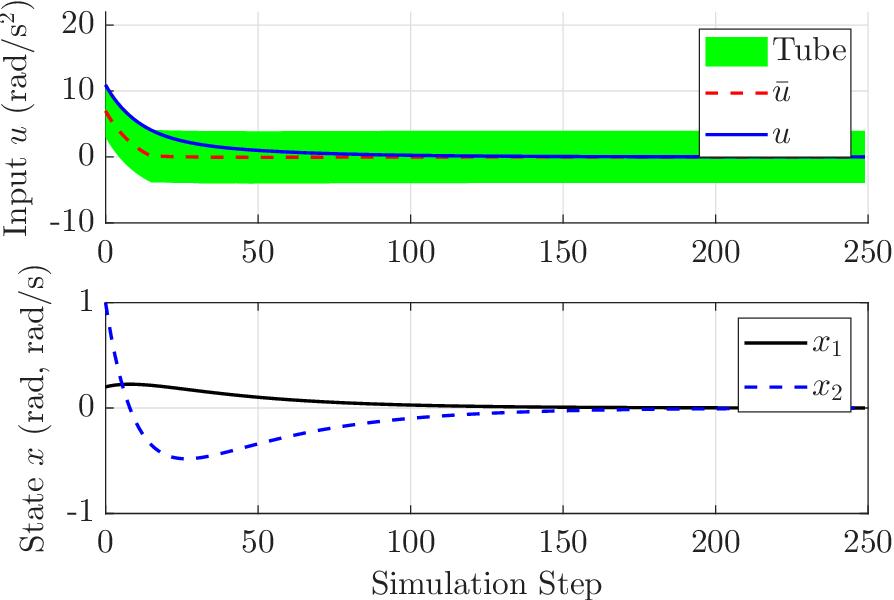}
  \includegraphics[width=0.95\linewidth]{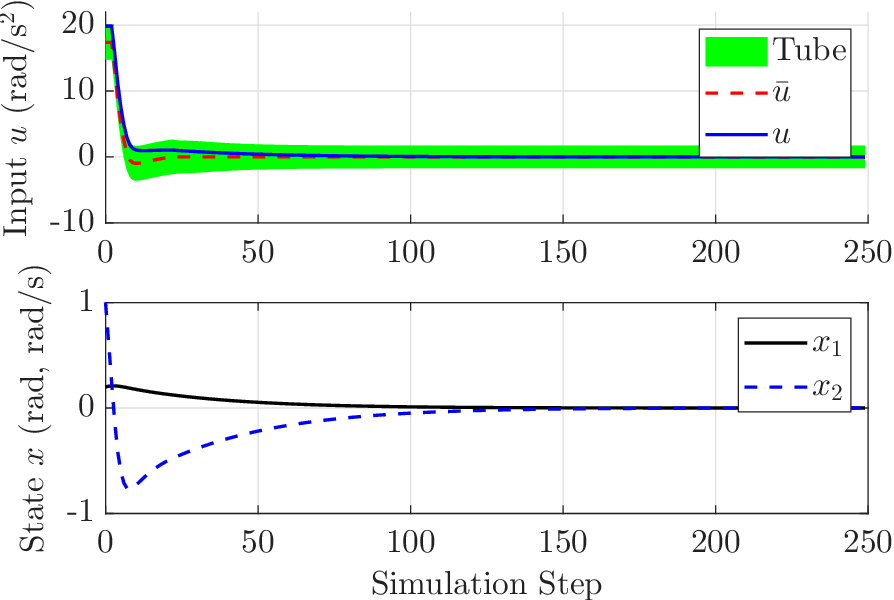}
  \caption{Comparison with baseline~\cite{Zhang+22}. Top: baseline method. Bottom: proposed method.}
  \label{fig:inverted_pendulum_comparison}
\end{figure}
For the baseline, we use the same observables $z$ as in Section 4.2 of the previous work~\cite{Zhang+22}, and the MPC optimization problem is solved by MOSEK\@;
all other settings are the same as above.
Unlike the baseline, our tube radius adapts more flexibly, allowing
near-maximum input usage from the early steps and achieving faster convergence.
However, the ultimate bound is conservative: at step $k=150$, the actual error is
$\|\x(150)-\xref\|=1.07\times 10^{-7}$, which is much smaller than
$\delta_f/\sqrt{\alpha_1}=0.540$.

%% file: 50_conclusion/conclusion.tex
This work proposed a robust control method for unknown nonlinear systems by
combining bilinear Koopman realizations and TMPC based on contraction
metrics.
We first learned a bilinear Koopman realization and estimated a modeling
error set from data.
We then formulated a discrete-time RCCM-based time-varying TMPC algorithm for the resulting
disturbed discrete-time control-affine system.
Numerical experiments on an inverted pendulum showed effectiveness of the
proposed method and faster convergence than an existing method.
One direction for future work is to reduce the conservatism of the
ultimate bound in~\cref{eq:state_convergence}.